\documentclass[12pt,a4paper,reqno]{amsart}
\allowdisplaybreaks
\usepackage{amsmath}
\usepackage{amsfonts}
\usepackage{amssymb,amsthm,amsfonts,amsthm,latexsym,enumerate,url,cases}
\numberwithin{equation}{section}
\usepackage{mathrsfs}
\usepackage{hyperref}
\hypersetup{colorlinks=true,citecolor=blue,linkcolor=blue,urlcolor=blue}
\usepackage{xcolor}
\usepackage{verbatim}
\usepackage{todonotes}
\usepackage{bbm}

\def\pmod #1{\ ({\rm{mod}}\ #1)}

\DeclareMathOperator\dif{d\!}

\theoremstyle{plain}
\newtheorem{thm}{Theorem}[section]

\newtheorem{lemma}[thm]{Lemma}
\newtheorem{problem}{Problem}

\newtheorem{conjecture}[thm]{Conjecture}
\theoremstyle{definition}

\usepackage{etoolbox}
\makeatletter
\patchcmd{\@settitle}{\uppercasenonmath\@title}{}{}{}
\patchcmd{\@setauthors}{\MakeUppercase}{}{}{}
\patchcmd{\section}{\scshape}{}{}{}
\makeatother

\begin{document}

\title
[The Diophantine Frobenius Problem revisited]
{The Diophantine Frobenius Problem revisited}

\author
[Y. Ding, W. Wang, H. Zhang] {Yuchen Ding, Weijia Wang and Hao Zhang}

\address{(Yuchen Ding) School of Mathematics,  Yangzhou University, Yangzhou 225002, People's Republic of China}
\email{ycding@yzu.edu.cn}

\address{(Weijia Wang) Morningside Center of Mathematics, Academy of Mathematics and Systems Science, Chinese Academy of
Sciences, Beijing 100190, People's Republic of China}
\email{weijiawang@amss.ac.cn}

\address{(Hao Zhang) School of Mathematics, Hunan University, Changsha 410082, People's Republic of China}
\email{zhanghaomath@hnu.edu.cn}

\subjclass[2010]{}

\begin{abstract}
Let $k\ge 2$ and $a_1, a_2, \cdots, a_k$ be positive integers with
$$
\gcd(a_1, a_2, \cdots, a_k)=1.
$$
It is proved that there exists a positive integer $G_{a_1, a_2, \cdots, a_k}$ such that every integer $n$
strictly greater than it can be represented as the form
$$
n=a_1x_1+a_2x_2+\cdots+a_kx_k, \quad (x_1, x_2, \cdots, x_k\in\mathbb{Z}_{\ge 0},~\gcd(x_1, x_2, \cdots, x_k)=1).
$$
We then investigate the size of  $G_{a_1, a_2}$ explicitly.
Our result strengthens the primality requirement of $x$'s in the classical Diophantine Frobenius Problem.
\end{abstract}
\maketitle

\section{Introduction}
Let $a_1, a_2, \cdots, a_k$ be a set of $k(\ge 2)$ positive integers with
$
\gcd(a_1, a_2, \cdots, a_k)=1.
$
It is well--known that all sufficiently large integers $n$ can be written as the form
\begin{align}\label{eq1}
n=a_1x_1+a_2x_2+\cdot\cdot\cdot+a_kx_k \quad (x_1, x_2, \cdots, x_k\in \mathbb{Z}_{\ge 0}),
\end{align}
where $\mathbb{Z}_{\ge 0}$ is the set of nonnegative integers. The Diophantine Frobenius Problem posed by Frobenius (see, e.g. \cite{RA}) asks the closed form of the minimal value $g_{a_1, a_2, \cdots, a_k}$ such that all integers $n>g_{a_1, a_2, \cdots, a_k}$ can be expressed as the form (\ref{eq1}). For $k=2$ Sylvester \cite{Sylvester} observed $g_{a_1,a_2}=a_1a_2-a_1-a_2$ and furthermore noticed that for any $0\le s\le g_{a_1,a_2}$ exactly one of $s$ and $g_{a_1,a_2}-s$ could be expressed as the desired form. For $k=3$, closed forms involving particular cases were extensively studied (see, e.g. \cite{RA}). We refer to the excellent monograph \cite{RA} of Ram\'{\i}rez Alfons\'{\i}n for a comprehensive literature on this problem.

In 2020, Ram\'{\i}rez Alfons\'{\i}n and Ska{\l}ba \cite{RS} made some considerations of the Diophantine Frobenius Problem in primes. Specifically, they were interested in the primes $p\le g_{a_1,a_2}$ with the form $a_1x_1+a_2x_2~(x_1,x_2\in \mathbb{Z}_{\ge 0})$. Suppose that $\pi_{a_1,a_2}$ is the number of such primes, then Ram\'{\i}rez Alfons\'{\i}n and Ska{\l}ba proved that for any $\varepsilon>0$ there is some constant $c_\varepsilon>0$ such that
\begin{align*}
\pi_{a_1,a_2}>c_\varepsilon \frac{g_{a_1,a_2}}{(\log g_{a_1,a_2})^{2+\varepsilon}}.
\end{align*}
The above inequality immediately deduces that $\pi_{a_1,a_2}>0$ for all sufficiently large $g_{a_1,a_2}$. Mathematical experiments then led them to the following conjecture \cite[Conjecture 2]{RS}.

\begin{conjecture}\label{conjecture1}
Let $2< a_1<a_2$ be two relatively prime integers. Then $\pi_{a_1,a_2}>0$.
\end{conjecture}

Let $\pi(t)$ be the number of primes up to $t$. On noting the antisymmetric property of the integers $n\le g_{a_1,a_2}$ of the form $a_1x_1+a_2x_2~(x_1,x_2\in \mathbb{Z}_{\ge 0})$, Ram\'{\i}rez Alfons\'{\i}n and Ska{\l}ba \cite{RS} also made another reasonable conjecture \cite[Conjecture 3]{RS}.
\begin{conjecture}\label{conjecture2}
Let $2< a_1 <a_2$ be two relatively prime integers, then
$$
\pi_{a_1,a_2}\sim\frac{\pi(g_{a_1,a_2})}{2}\quad (\text{as}~a_1\rightarrow\infty).
$$
\end{conjecture}

Recently, Ding \cite{Ding} and Ding, Zhai and Zhao \cite{DingZhaiZhao} proved Conjecture \ref{conjecture2}.  In a more recent article, Dai, Ding and Wang \cite{DZZ} confirmed Conjecture \ref{conjecture1}. In \cite{Chen-Zhu1, Chen-Zhu2}, Chen and Zhu obtained  further results on primes of the form $ax+by$.

The motivation of this note is the following observation from Conjecture \ref{conjecture1}. The validity of it clearly means that there exists a prime $p<g_{a_1,a_2}$ of the form
\begin{align}\label{eq2}
p=a_1x_1+a_2x_2 \quad (x_1, x_2\in \mathbb{Z}_{\ge 0}).
\end{align}
Moreover, the integers $x_1$ and $x_2$ in (\ref{eq2}) must satisfy $\gcd(x_1,x_2)=1$. This naturally leads us to ask whether all sufficiently large integers $n$ can be written in the form
\begin{align}\label{eq3}
n=a_1x_1+a_2x_2, \quad (x_1, x_2\in\mathbb{Z}_{\ge 0},~\gcd(x_1, x_2)=1).
\end{align}
If the answer is affirmative, let $G_{a_1,a_2}$ be the least integer such that all integers $n>G_{a_1, a_2}$ can be expressed in the form (\ref{eq3}). We are going to show that $G_{a_1,a_2}$ is indeed well defined. Generally, we can extend $G_{a_1,a_2}$ to $k$ variables. Let $a_1, a_2, \cdots, a_k$ be positive integers with
$
\gcd(a_1, a_2, \cdots, a_k)=1.
$
Let $G_{a_1, a_2, \cdots, a_k}$ be the least integer such that all integers $n>G_{a_1, a_2, \cdots, a_k}$ can be expressed as the form
$$
n=a_1x_1+a_2x_2+\cdots+a_kx_k, \quad (x_1, x_2, \cdots, x_k\in\mathbb{Z}_{\ge 0},~\gcd(x_1, x_2, \cdots, x_k)=1).
$$
The finiteness fact of $G_{a_1, a_2, \dots, a_k}$ for general $k$ can also be proved.

\begin{thm}\label{new-theorem-2}
Let $k\ge 2$ and $a_1, a_2, \cdots, a_k$ be positive integers with
$$
\gcd(a_1, a_2, \cdots, a_k)=1.
$$
Then $G_{a_1,\dots,a_k}$ is finite.
\end{thm}

We are now in a position to highlight the title of this article.

\begin{problem}[The Diophantine Frobenius Problem revisited]\label{pro2}
Let $a_1, a_2, \cdots, a_k$ be positive integers with $\gcd(a_1, a_2, \cdots, a_k)=1$. Determine the closed form of $G_{a_1, a_2, \cdots, a_k}$.
\end{problem}

From now on, we will focus on the investigations of two variables situations.

Let $\omega(n)$ be the number of different prime factors of $n$ and $\varphi(n)$ the Euler totient function. Let $\{t\}=t-\lfloor t\rfloor$ be the fractional part of $t$.
Let $1< a_1<a_2$ be two relatively prime integers. For a positive integer $n$ let
$$
f(n)=\#\big\{(x_1, x_2)\in \mathbb{Z}_{\ge 0}^2: a_1x_1+a_2x_2=n, \gcd(x_1, x_2)=1\big\}.
$$
By this notation, we have $f(n)>0$ for any $n>G_{a_1,a_2}$.  Using similar arguments of Theorem \ref{new-theorem-2} we can give the following closed form of $f(n)$.

\begin{thm}\label{new-thm1}
	Let $1< a_1<a_2$ be two integers with $\gcd(a_1, a_2)=1$. Suppose that $0\leq r_n<a_1$ denotes the unique integer such that $a_2r_n\equiv n\pmod{a_1}$. Then we have
	\[
    f(n)=\frac{\varphi(n)}{a_1a_2}+E(n),
	\]
	where the error term satisfies $|E(n)|< 2^{\omega(n)}$ having the explicit expression
    $$
    E(n)=\sum_{d\mid n}\mu\left(\frac{n}{d}\right)\bigg(1-\frac{r_{d}}{a_1}-\Big\{\frac{d-a_2r_d}{a_1a_2}\Big\}\bigg).
    $$
\end{thm}

We now take a close look at the error term $E(n)$ in Theorem \ref{new-thm1}.
It is well-known that (see e.g., \cite[Page 238, 5(b)]{Montgomery}) there is a positive constant $c$ such that
$$
\sum_{n\leq N}2^{\omega(n)}=cN\log N+O(N),
$$
from which it follows that
 \[\sum_{n\leq N}E(n)\ll \sum_{n\leq N}2^{\omega(n)}\ll N\log N.\]
It seems interesting to improve the above trivial bound involving the mean value of $E(n)$. We are able to give a conditional improvement of it. The results on zero-free region of $\zeta(s)$ at present does not seem possible to provide an unconditional improvement by the same argument of the following theorem.

\begin{thm}\label{thm:meanvalue}
    Assuming the generalized Riemann hypothesis, for any $\varepsilon>0$ we have 
    \[\sum_{n\leq N}E(n)\ll a_1a_2N^{\frac{1}{2}+\varepsilon},\]
    where the implied constant depends only on $\varepsilon$.  
\end{thm}

As an application of Theorem \ref{new-thm1}, for any $\varepsilon>0$ we have
$$
G_{a_1, a_2}\ll_\varepsilon a_1a_2\exp\left(\frac{(\log 2+\varepsilon)\log (a_1a_2)}{\log\log (a_1a_2)}\right)
$$
from explicit bounds of $\omega(n)$ (see \cite[Theorem 12]{Robin}) and $\varphi(n)$ (see \cite[Theorem 15]{Rosser}) with rountine computations.
We will obtain more explicit estimations of $G_{a_1, a_2}$.

\begin{thm}\label{thm:upperbound}
Let $1< a_1<a_2$ be two integers with $\gcd(a_1, a_2)=1$. Then we have
$$
a_1a_2\le G_{a_1, a_2}\ll a_1a_2(\log a_1a_2)^2,
$$
where the implied constant is absolute.
\end{thm}

\begin{thm}\label{thm:lowerbound}
    Let $a_1>2$ be a given integer. Then there is a positive constant $c_1$ depending only on $a_1$ such that
    \[
\limsup_{\substack{a_2\to\infty\\ \gcd(a_1,a_2)=1}}\frac{G_{a_1,a_2}}{a_2\log a_2 }>c_1.
\]
\end{thm}

For fixed $a_1$, there is a small distance between the maximal orders of $G_{a_1,a_2}$ obtained by Theorems \ref{thm:upperbound} and \ref{thm:lowerbound}. Determining the exact maximal order of $G_{a_1,a_2}$ is an unsolved problem.

It is easy to see that the values of  $g_{a_1,a_2}$ are always odd. Mathematical experiments indicate that most values of $G_{a_1, a_2}$ are even. At present, we have no idea what kind of mathematical logic lies behind this. We are able to calculate a few values of $G_{a_1, a_2}$ up to $1<a_1<a_2\le 200$ with $\gcd(a_1, a_2)=1$ and see that all of them are even, except for
$$
G_{4, 13}=231,~ G_{12, 13}=693,~ G_{10, 37}=1653,~ G_{23,29}=3927,
$$
\[G_{28,95}=23205,~G_{7,83}=3705,~G_{7,90}=3705,~G_{10,199}=11571,\]
\[G_{24,199}=42315,~G_{29,180}=49665,~G_{29,189}=58695,~G_{49,160}=64155,\]
\[G_{49,171}=73185, ~G_{89, 133}=123585,~ G_{72, 199}=126945.\]
Here comes another interesting point, involving the parity of the value of $G_{a_1, a_2}$.

\begin{problem}\label{pro1}
Does $G_{a_1, a_2}$ take infinitely many odd values often?
\end{problem}

Unfortunately, we cannot answer this at present. However, we are able to prove that $G_{a_1, a_2}$ take even values infinitely many times. Actually, this fact follows from the following more explicit result.

\begin{thm}\label{thm2}
Let $a$ be an odd integer greater than $2$. Then
$
G_{2, a}=4a-2.
$
\end{thm}
Comparing Theorems \ref{thm:lowerbound} and \ref{thm2} we see that the growth of $G_{a_1,a_2}$ shows strikingly different features depending on whether $a_1=2$ or not.

For fixed $a_1>2$, we have $\gcd(ka_1\pm 1,a_1)=1$ for any positive integer $k$. Thus, there are infinitely many $a_2$ such that both $a_2$ and $a_2+2$ are relatively prime with $a_1$.  We do not know the answer to the following problem which is in the fashion of the Chebyshev bias phenomenon \cite{Chebyshev}.

\begin{problem}\label{pro1-2-new}
For any fixed $a_1>2$, does the sign of $G_{a_1, a_2+2}-G_{a_1, a_2}$ change infinitely many often?
\end{problem}

\section{Proofs of Theorem \ref{new-theorem-2} and Theorem \ref{new-thm1}}

\begin{proof}[Proof of Theorem \ref{new-theorem-2}]
	For any positive integer $n$, we define
	$$
	f_{a_1,\ldots,a_k}(n)=\#\big\{(x_1,\ldots, x_k)\in \mathbb{Z}_{\ge 0}^k: a_1x_1+\cdots+a_kx_k=n, \gcd(x_1,\ldots, x_k)=1\big\},
	$$
	and
	$$
	g_{a_1,\ldots,a_k}(n)=\#\big\{(x_1,\ldots, x_k)\in \mathbb{Z}_{\ge 0}^k: a_1x_1+\cdots+a_kx_k=n\big\}.
	$$
   Note that if $d=\gcd(x_1,\ldots, x_k)$, then clearly we have $d|n$, which leads to
   \begin{align*}
       g_{a_1,\ldots,a_k}(n)&=\sum_{d\mid n}\#\big\{(x_1,\ldots, x_k)\in \mathbb{Z}_{\ge 0}^k: a_1x_1+\cdots+a_kx_k=n, \gcd(x_1,\ldots, x_k)=d\big\}\\
       &=\sum_{d\mid n}f_{a_1,\ldots,a_k}\left(\frac{n}{d}\right).
   \end{align*}
	Then by the M\"obius inversion formula (see e.g., \cite[Theorem 2.9]{Apostol}) we have
	\begin{equation}\label{eq:fa1ak}
	    f_{a_1,\ldots,a_k}(n)=\sum_{d\mid n}\mu(d)g_{a_1,\ldots,a_k}\Big(\frac{n}{d}\Big),
	\end{equation}
   where $\mu(n)$ is the M\"obius function..
    On the other hand, by  \cite[Eq. (1.3)]{Bell}, we see that 
    $$
    g_{a_1,\dots,a_k}(n)=c_0+c_1n+\cdots+c_{k-1}n^{k-1}
    $$ 
    is a polynomial in $n$ of degree $k-1$ with rational coefficients $c$'s which are independent of $n$. Note that $g_{a_1,\dots,a_k}(n)>0$ for $n>g_{a_1,\dots,a_k}$, which clearly means that $c_{k-1}>0$.  So by combining \eqref{eq:fa1ak}, we know that 
    \begin{align}\label{eq-9-1-1}
    f_{a_1,\ldots,a_k}(n)=&
    c_{k-1}\sum_{d\mid n}\mu(d)\left(\frac nd\right)^{k-1}+O\bigg(n^{k-2}\sum_{\substack{d\mid n,~ \mu(d)\neq 0}} 1\bigg),\nonumber\\
    =&c_{k-1}n^{k-1}\sum_{d\mid n}\frac{\mu(d)}{d^{k-1}}+O\Big(n^{k-2}2^{\omega(n)}\Big).
    \end{align}
For $k\ge 3$ it is clear that
    $$
    \sum_{d\mid n}\frac{\mu(d)}{d^{k-1}}=\prod_{p\mid n}\left(1-\frac 1{p^{k-1}}\right)>\rho_k,
    $$
    where $\rho_k>0$ is a constant depending only on $k$. While for $k=2$, one notes that 
    $$
    n \sum_{d\mid n}\frac{\mu(d)}{d}=n\prod_{p\mid n}\Big(1-\frac{1}{p}\Big)=\varphi(n)
    $$
    and $\varphi(n)/2^{\omega(n)}\rightarrow\infty$ as $n\rightarrow\infty$.
    Thus, in both cases we have $f_{a_1,\ldots,a_k}(n)>0$ for all sufficiently large $n$  from (\ref{eq-9-1-1}), proving our theorem.
\end{proof}

For the proof of Theorem \ref{new-thm1}, we make use of the following explicit formula of $g(n)$.

\begin{lemma}\label{lemma4}
	Let $1< a_1<a_2$ be two relatively prime integers and $n$ a positive integer. Suppose that $0\leq r_n<a_1$ denotes the unique integer such that $a_2r_n\equiv n\pmod{a_1}$. Then we have
	\[
	g(n)=\left\lfloor\frac{n-a_2r_n}{a_1a_2}\right\rfloor+1.
	\]
\end{lemma}
\begin{proof}
	Since $a_2r_n\equiv n\pmod{a_1}$ we can assume
	$
	n=a_1y_n+a_2r_n
	$
	for some integer $y_n$.
	The arguments will be separated into two cases.
	
	{\bf Case I.} $g(n)=0$. In this case, we clearly have $y_n<0$ which implies $n<a_2r_n$. Hence,
	$$
	\left\lfloor\frac{n-a_2r_n}{a_1a_2}\right\rfloor+1=-1+1=0=g(n).
	$$
	
	{\bf Case II.} $g(n)\ge 1$. In this case, we have $y_n\ge 0$ and
	$$
	n=a_1(y_n-\ell a_2)+a_2(y_2+\ell a_1),
	$$
	for any integer $\ell$ satisfying $0\le \ell\le y_n/a_2$. It then follows that
	$$
	g(n)=\left\lfloor \frac{y_n}{a_2}\right\rfloor+1=\left\lfloor\frac{n-a_2r_n}{a_1a_2}\right\rfloor+1,
	$$
	which completes the proof of Lemma \ref{lemma4}.
\end{proof}

\begin{proof}[Proof of Theorem \ref{new-thm1}]	
	 Let $n$ be a positive integer. By (\ref{eq:fa1ak}) with $k=2$ we have
	\begin{align}\label{equation2}
		f(n)=\sum_{d\mid n}\mu(d)g\Big(\frac{n}{d}\Big).
	\end{align}
	We see from Lemma \ref{lemma4} that
	\begin{align}\label{equation4}
		g(n)=\left\lfloor\frac{n-a_2r_n}{a_1a_2}\right\rfloor+1=\frac{n}{a_1a_2}+R(n),
	\end{align}
	where $r_n$ is an integer satisfying $0\leq r_n<a_1$ and
	$$
	|R(n)|=\left|1-\frac{r_n}{a_1}-\Big\{\frac{n-a_2r_n}{a_1a_2}\Big\}\right|\le 1.
	$$
	 Now, inserting (\ref{equation4}) into (\ref{equation2}) leads to our desired result.
\end{proof}

\section{Proof of Theorem \ref{thm:meanvalue}}
Theorem \ref{thm:meanvalue} is contained in the following more general theorem as a simple case.

\begin{thm}\label{thm:general}
    Let $A, q>0$ be two fixed numbers. Suppose that $h(n)$ is a periodic function over $\mathbb{Z}/q\mathbb{Z}$ with $|h(n)|\le A$. Then assuming the generalized Riemann hypothesis, for any $\varepsilon>0$ we have
    $$
    \sum_{n\le N}\sum_{d\mid n}\mu\left(\frac{n}{d}\right)h(d)\ll A q N^{1/2+\varepsilon},
    $$
    where the implied constant depends only on $\varepsilon$.
\end{thm}

We first point out that how Theorem \ref{thm:general} implies Theorem \ref{thm:meanvalue}.

\begin{proof}[Proof of Theorem \ref{thm:meanvalue} via Theorem \ref{thm:general}]
In the present case, it can be easily seen that 
$$
h(n)=1-\frac{r_{n}}{a_1}-\Big\{\frac{n-a_2r_n}{a_1a_2}\Big\},
$$
is a periodic function over $\mathbb{Z}/a_1a_2\mathbb{Z}$ with $|h(n)|\le 1$. Now, Theorem \ref{thm:meanvalue} follows from Theorem \ref{thm:general} with $q=a_1a_2$ and $A=1$.
\end{proof}
 
 Let $\alpha(s)=\sum_{n=1}^{\infty}a(n) n^{-s}$ be a Dirichlet series and $\sigma_a$ be the abscissa of convergence of the series $\sum_{n=1}^{\infty}|a(n)| n^{-s}$. The proof of Theorem \ref{thm:general} is an application of Perron's formula (see e.g., \cite[Theorem 5.2 and Corollary 5.3]{Montgomery}).

\begin{lemma}[Perron's formula]\label{perron}
If $\sigma_0>\max\{0,\sigma_a\}$ and $x>0$ is not an integer, then
$$
\sum_{n\le x}a(n)=\frac{1}{2\pi i}\int_{\sigma_0-iT}^{\sigma_0+iT}\alpha(s)\frac{x^s}{s}ds+R,
$$
where
$$
R\ll \sum_{\substack{x/2<n<2x}}|a(n)|\min\left\{1, \frac{x}{T|x-n|}\right\}+\frac{4^{\sigma_0}+x^{\sigma_0}}{T}\sum_{n=1}^{\infty}\frac{|a(n)|}{n^{\sigma_0}}.
$$
\end{lemma}

Let $s=\sigma+it$ and $\tau=|t|+4$.
 The following technical result involving the Reimann $\zeta$ function is standard in analytical number theory, see e.g., \cite[Theorem 13.23]{Montgomery}. 

\begin{lemma}\label{littlewood}
    Let $\varepsilon>0$ be arbitrarily small. Assuming the Riemann hypothesis, there is a constant $c_\varepsilon>0$ such that for all $\sigma\ge 1/2+\varepsilon$ and $|t|\ge 1$ we have
    $$
    \left|\frac{1}{\zeta(s)}\right|\le \exp\left(\frac{c_\varepsilon\log \tau}{\log\log \tau}\right).
    $$  
\end{lemma}
Lemma \ref{littlewood} is a quantitative form of the Lindel\"{o}f hypothesis which was obtained by Littlewood in 1912. Parallel to Lemma \ref{littlewood}, we have the following bound of $L$-function, see e.g., \cite[Page 445, Exercise 8]{Montgomery}.

\begin{lemma}\label{littlewood-2}
   Let $\chi$ be a primitive Dirichlet character modulo $q$ with $q>1$, and suppose that $L(s,\chi)\neq 0$ for $\sigma>1/2$. Then there is an absolute constant $c>0$ such that
   $$
   |L(s, \chi)|\le \exp\left(\frac{c\log q\tau}{\log\log q\tau}\right),
   $$
   uniformly for $1/2\le \sigma\le 3/2$.
\end{lemma}

\begin{proof}[Proof of Theorem \ref{thm:general}]
By orthogonality of characters we have
\begin{align}\label{perron-eq-2}
h(n)=\sum_{k\mid q}\mathbbm{1}_{k=\gcd(n,q)}(n)\sum_{\chi_k \pmod{\frac{q}{k}}}c_{k, \chi}~\chi_k\left(\frac{n}{k}\right),
\end{align}
where the second sum of $\chi_k$ above runs through all the Dirichlet characters mod $q/k$, and the coefficients $c_{k, \chi}$ are given by
$$
c_{k, \chi}=\frac{1}{\varphi(q/k)}~\sideset{}{^*}\sum_{m \pmod{\frac{q}{k}}}h(km)\overline{\chi}_k(m).
$$
Here, the sum of $m$ runs through the reduced residue system mod $q/k$.

For large $N$ let $N_1=N+1/2$. By (\ref{perron-eq-2}) we have
\begin{align}\label{perron-eq-3}
\sum_{n\le N_1}\sum_{d\mid n}\mu\left(\frac{n}{d}\right)h(d)&=\sum_{k\mid q}\sum_{\chi_k \pmod{\frac{q}{k}}}c_{k, \chi}\sum_{n\le N_1}\sum_{\substack{d\mid n\\ k\mid  d}}\mu\left(\frac{n}{d}\right)\chi_k\left(\frac{d}{k}\right)\mathbbm{1}_{k=\gcd(d,q)}(d)\nonumber\\
&=\sum_{k\mid q}\sum_{\chi_k \pmod{\frac{q}{k}}}c_{k, \chi}\sum_{n\le N_1}\sum_{\substack{d\mid n\\ k\mid  d}}\mu\left(\frac{n}{d}\right)\chi_k\left(\frac{d}{k}\right)\mathbbm{1}_{1=\gcd\left(\frac{d}{k},\frac{q}{k}\right)}(d)\nonumber\\
&=\sum_{k\mid q}\sum_{\chi_k \pmod{\frac{q}{k}}}c_{k, \chi}\sum_{n\le N_1}\sum_{\substack{d\mid n\\ k\mid  d}}\mu\left(\frac{n}{d}\right)\chi_k\left(\frac{d}{k}\right)\nonumber\\
&=\sum_{\substack{k\mid q\\ k<q}}\sum_{\chi_k \pmod{\frac{q}{k}}}c_{k, \chi}\sum_{n\le N_1}a_k(n)+1,
\end{align}
where 
$$
a_k(n)=\sum_{\substack{d\mid n\\ k\mid  d}}\mu\left(\frac{n}{d}\right)\chi_k\left(\frac{d}{k}\right)
$$
and the term $k=q$ contributes $1$ in (\ref{perron-eq-3}) because $\chi_q(n)=1$ for all $n$ and
\begin{align*}
\sum_{\substack{d\mid n\\ q\mid  d}}\mu\left(\frac{n}{d}\right)\chi_{q}\left(\frac{d}{q}\right)=\sum_{\substack{d\mid n\\ q\mid  d}}\mu\left(\frac{n}{d}\right)=
\begin{cases}
1, & \text{if~}n=1,\\
0, &\text{otherwise.}
\end{cases}
\end{align*}
We are leading to estimate the sum
$
S_k(N):=\sum_{n\le N_1}a_k(n).
$
Let 
$$
\alpha_k(s)=\sum_{n=1}^{\infty}\frac{a_k(n)}{n^s}.
$$
On making $n=kh$ and $d=k\ell$, for $\Re s>1$ we have
\begin{align*}
\alpha_k(s)=\frac{1}{k^s}\sum_{\substack{h=1}}^{\infty}h^{-s}\sum_{\ell\mid h}\mu\left(\frac{h}{\ell}\right)\chi_k(\ell)=\frac{1}{k^s}\frac{L(s,\chi_k)}{\zeta(s)},
\end{align*}
where $L(s, \chi_k)=\sum_{n=1}^{\infty}\frac{\chi_k(n)}{n^s}$ is the Dirichlet $L$ function attached to the character $\chi_k$. 
The function $\alpha$ is naturally analytically continued to other points on the complex plane by the functions $\zeta$ and $L$.

By Lemma \ref{perron}, for $\sigma_0>1$ we have
\begin{align}\label{eq-9-2-1}
\sum_{n\le N_1}a_k(n)=\frac{1}{2\pi i}\int_{\sigma_0-iT}^{\sigma_0+iT}\frac{N_1^sL(s, \chi_k)}{s\zeta(s)k^s}\dif s+R,
\end{align}
where
$$
R\ll \sum_{\substack{N_1/2<n<2N_1}}|a_k(n)|\min\left\{1, \frac{N_1}{T|N_1-n|}\right\}+\frac{(4N_1)^{\sigma_0}}{T}\sum_{n=1}^{\infty}\frac{|a_k(n)|}{n^{\sigma_0}}.
$$
By the bound of $\omega(n)$ \cite[Theorem 12]{Robin}, we have
$$
|a_k(n)|\le 2^{\omega(n)}<2^{\frac{2\log n}{\log\log n}}\le \exp\left(\frac{2\log N_1}{\log\log N_1}\right).
$$
Hence for $2\le T\le N_1$ we get
\begin{align}\label{eq-9-2-2}
    R&\ll \exp\left(\frac{2\log N_1}{\log\log N_1}\right)\left(1+\frac{N_1}{T}\sum_{n\le N_1}\frac 1n\right)+\frac{N_1}{T}\log N_1\nonumber\\
    &\ll \exp\left(\frac{2\log N_1}{\log\log N_1}\right)\frac{N_1}{T}\log N_1,
\end{align}
by appointing $\sigma_0=1+\frac{1}{\log N_1}$, where the implied constants are absolute.

 For any $\varepsilon>0$ let $\sigma_1=\frac12+\varepsilon$. Throughout our proof, $\varepsilon$ may be different at different occasions. Let also $\mathscr{C}$ be the closed contour that consists of line segments joining the points $\sigma_0-iT, \sigma_0+iT, \sigma_1+iT$ and $\sigma_1-iT$.  The famous Riemann hypothesis states that all zeros of $\zeta(s)$ in the critical strip $0\le \Re s\le 1$ lie on the critical line $\Re s=1/2$. It is also well-known that $L(s, \chi_k)$ is an analytic function over the complex plane. Hence, the function
 $
\frac{N_1^sL(s, \chi_k)}{s\zeta(s)k^s}
 $ 
 is analytic inside the counter $\mathscr{C}$, and by the Cauchy residue theorem we have
 \begin{align}\label{eq-9-2-3}
 \frac{-1}{2\pi i}\int_{\mathscr{C}}\frac{N_1^sL(s, \chi_k)}{s\zeta(s)k^s}\dif s=0.
 \end{align}
Noting that $k<q$, the modulus of $\chi_k$ is $\frac{q}{k}>1$. Moreover, if $\chi_k$ is principle, then 
$$
L(s,\chi_k)=\prod_{p\mid  k}\left(1-\frac{1}{p^s}\right)\zeta(s).
$$
Therefore, by Lemmas \ref{littlewood} and \ref{littlewood-2} we have
\begin{align*}
    \frac{1}{2\pi i}\int_{\sigma_1+iT}^{\sigma_1-iT} \frac{N_1^sL(s, \chi_k)}{s\zeta(s)k^s}\dif s 
    \ll N_1^{1/2+\varepsilon}\bigg(\int_{-T}^{T} \frac{(k\tau)^{\varepsilon}\tau^\varepsilon}{k^{\sigma_1}\sqrt{t^2+\sigma_1^2}}\dif t+\frac1{k^{\sigma_1}}\bigg)
    \ll\frac{1}{\sqrt{k}}N_1^{1/2+\varepsilon}T^{\varepsilon},
\end{align*}
where the implied constants depend only on $\varepsilon$. Again, by Lemmas \ref{littlewood} and \ref{littlewood-2} we have
\begin{align*}
    \frac{1}{2\pi i}\bigg(\int_{\sigma_1-iT}^{\sigma_0-iT}+\int_{\sigma_0+iT}^{\sigma_1+iT}\bigg)\frac{N_1^sL(s, \chi_k)}{s\zeta(s)k^s}\dif s
    \ll \frac{1}{\sqrt{k}}N_1T^{-1+\varepsilon},
\end{align*}
where the implied constant depends only on $\varepsilon$. We now conclude from the above estimates that
\begin{align*}
\sum_{n\le N_1}a_k(n)\ll \frac{1}{\sqrt{k}}\left(\frac{N_1^{1+\varepsilon}}{T}+N_1^{1/2+\varepsilon}T^{\varepsilon}\right),
\end{align*}
in view of (\ref{eq-9-2-1}), (\ref{eq-9-2-2}) and (\ref{eq-9-2-3}). Taking $T=N_1^{1/2}$ we get
\begin{align*}
\sum_{n\le N_1}a_k(n)\ll \frac{1}{\sqrt{k}}N_1^{1/2+\varepsilon}.
\end{align*}
Inserting this into (\ref{perron-eq-3}), we have
\begin{align*}
\sum_{n\le N}E(n)\ll N^{1/2+\varepsilon} \sum_{\substack{k\mid q\\ k<q}}~\sideset{}{^*}\sum_{\chi_k \pmod{\frac{q}{k}}}\frac{|c_{k, \chi}|}{\sqrt{k}}+1,
\end{align*}
where the implied constant depends only on and $\varepsilon$. Since $|h(n)|\le A$, we know that
$
|c_{k, \chi}|\le A,
$
from which it clearly follows that
\begin{align*}
\sum_{n\le N}E(n)\ll A N_1^{1/2+\varepsilon} \sum_{k\mid q}\varphi\left(\frac{q}{k}\right)\frac{1}{\sqrt{k}}\ll Aq N^{1/2+\varepsilon},
\end{align*}
where the implied constants depend only on and $\varepsilon$.
\end{proof}

\section{Proofs of Theorem \ref{thm:upperbound} and Theorem \ref{thm:lowerbound}}
We now proceed to the proof of theorem \ref{thm:upperbound}.

{\bf Lower bound of $G_{a_1, a_2}$.}
One easily notes that $a_1a_2$ cannot be represented as the desired form. To see this, we assume the contrary, i.e.,
$$a_1a_2=a_1x_1+a_2x_2, \quad (x_1, x_2\in\mathbb{Z}_{\ge 0},~\gcd(x_1, x_2)=1).$$
Then we have $a_2\mid (a_2-x_1)$. Thus, $x_1=0$ or $a_2$ which is a contradiction.

{\bf Upper bound of $G_{a_1, a_2}$.} For the proof of upper bound, the famous object Jacobsthal function $j(n)$ now comes into the play. The Jacobsthal function $j(n)$ is defined as the minimal integer, such that any $j(n)$ consecutive integers contain at least one integer which is coprime with $n$. For our applications, we need an alternative definition. Let
$\mathcal{P}_n$ be the set of different prime factors of $n$. For any $p\in \mathcal{P}_n$, we fix an integer $c_p$, and hence we form the set
$$
\mathcal{C}=\{c_p:p\in \mathcal{P}_n\}.
$$
The generalized Jacobsthal function $j_{\mathcal{C}}(n)$ is defined as the minimal integer, such that any $j_{\mathcal{C}}(n)$ consecutive integers contain at least one integer $m$ satisfying 
$$
m\not\equiv c_p\pmod{p},
$$
for all $p\in \mathcal{P}_n$. Clearly, $j_{\mathcal{C}}(n)$ reduces to $j(n)$ if all the $c_p$ are chosen to be $0$. The following lemma is an application of the Chinese Remainder Theorem.

\begin{lemma}\label{lemma-jacob}
For any given $\mathcal{C}$, we have $j_{\mathcal{C}}(n)\le j(n)$.
\end{lemma}
\begin{proof}
For any $j<j_{\mathcal{C}}(n)$, there exists an nonnegative integer $m$ such that for any $1\le i\le j$ there corresponds a prime factor $p_i$ of $n$ satisfying
$
m+i\equiv c_{p_i}\pmod{p_i}.
$
By the Chinese Remainder Theorem, there is a positive integer $K$  such that $K\equiv -c_p\pmod{p}$ for any $p\mid n$. We now consider the $j$ consecutive integers $m+K+1,\ldots, m+K+j$. Clearly, for any $1\le i\le j$ we have
$$
m+K+i\equiv c_{p_i}+(-c_{p_i})\equiv 0 \pmod{p_i}.
$$
Thus, by the definition we have $j(n)>j$, or $j(n)\ge j_{\mathcal{C}}(n)$.
\end{proof}

The following bound of $j(n)$ due to Iwaniec \cite{Iwaniec} is very famous in analytic number theory as the Jacobsthal function $j(n)$ lies in the heart of construction of large gaps between consecutive primes.
\begin{lemma}\label{lemma-Iwaniec}
We have $j(n)\ll (\log n)^2$, where the implied constant is abosulte.
\end{lemma}

\begin{proof}[Proof of the upper bound of $G_{a_1,a_2}$]
By Lemma \ref{lemma4} there are precisely $g(n)$ nonnegative integer solutions of $n=a_1x+a_2y$ which are
\begin{align}\label{residue}
\begin{cases}
    x=x_0-ka_2,\\
    y=y_0+ka_1,
\end{cases}    
\end{align}
where $0\le y_0<a_1$ satisfies $a_2y_0\equiv n\pmod{a_1}$ and $k=0, 1,\ldots, \left\lfloor\frac{n-a_2y_0}{a_1a_2}\right\rfloor+1$. In other words, there are at least $\left\lfloor n/(a_1a_2)\right\rfloor$ such $k$. If $\gcd(x,y)\neq 1$, then there is a prime factor $p$ of $n$ such that $p\mid  x$ and $p\mid  y$. Since $\gcd(a_1,a_2)=1$, we will separate the following arguments into three cases.

{\it Case I.} $p\nmid a_1$ and $p\nmid a_2$. In this case, by (\ref{residue}) we have
$$
k\equiv a_2^{-1}x_0\equiv -a_1^{-1}y_0:\equiv c_p\pmod{p}.
$$

{\it Case II.} $p\nmid a_1$ and $p\mid   a_2$. In this case, by (\ref{residue}) we have
$$
k\equiv -a_1^{-1}y_0:\equiv c_p\pmod{p}.
$$

{\it Case III.} $p\mid   a_1$ and $p\nmid a_2$. In this case, by (\ref{residue}) we have
$$
k\equiv a_2^{-1}x_0:\equiv c_p\pmod{p}.
$$
Now, we choose the set $\mathcal{C}$ to be $\{c_p:p\mid n\}$. Then any consecutive integers of lengeth $j_{\mathcal{C}}(n)$ contains at least one $k$ such that
$
k\not\equiv c_p\pmod{p}.
$
For such a $k$ we must have $\gcd(x,y)=1$, which means that if 
$
\left\lfloor n/(a_1a_2)\right\rfloor\ge j_{\mathcal{C}}(n)
$
then there exists some $k$ in (\ref{residue}) satisfying $\gcd(x,y)=1$. We now conclude from Lemmas \ref{lemma-jacob} and \ref{lemma-Iwaniec} that if 
\begin{align}\label{wang-lemma}
    \frac{n}{a_1a_2}\gg (\log n)^2,
\end{align}
then there is an expression of $n$ satisfying our requirement. From (\ref{wang-lemma}) it clearly follows that $G_{a_1a_2}\ll a_1a_2(\log a_1a_2)^2$.
\end{proof}

\begin{proof}[Proof of theorem \ref{thm:lowerbound}]
For any given $a_1>2$, we have $\varphi(a_1)\ge 2$. Thus, there are infinitely many primes $q$ such that
$$
q\not\equiv -1\pmod{a_1},
$$
thanks to Dirichlet's theorem in arithmetic progressions (see, e.g. \cite[Theorem 7.9]{Apostol}).
    It suffices to prove that for a given large prime $q>a_1$ with $a_1\nmid (q+1)$, we can find a suitable $a_2$ with $a_2\equiv -1\pmod{a_1}$ such that $G_{a_1,a_2}>qa_1a_2$ and $q\gg_{a_1} \log a_2$.
   For $a_2\equiv -1\pmod{a_1}$ it can be easily checked that a nonnegative solution of $a_1x+a_2y=qa_1a_2+1$ is
   \begin{align*}
   \begin{cases}
   x=\frac{a_2+1}{a_1}+(q-1)a_2,\\
   y=a_1-1.
   \end{cases}
   \end{align*}
 Hence, all the nonnegative solutions of $a_1x+a_2y=qa_1a_2+1$ are
    \begin{align}\label{7-19-1}
    x_{\ell}=\frac{a_2+1}{a_1}+(q-\ell)a_2,~ \ y_{\ell}=\ell a_1-1,\quad  \ell=1,2,\dots,q.
    \end{align}
    We will construct a suitable $a_2$ with $a_2\equiv -1\pmod{a_1}$ such that  $\gcd(x_{\ell},y_{\ell})>1$ for all $\ell=1,2,\dots,q$ by Chinese Remainder Theorem, from which our theorem follows.

    Since $q>a_1$ is prime, we see that there is exactly one positive integer in $[1,q]$, say $\ell_0$ such that $q\mid \ell_0a_1-1.$ Let $y_{\ell_0}=\ell_0a_1-1$. Since $q\not\equiv -1\pmod{a_1}$ by our choice of $q$, we see that $\ell_0a_1-1\neq q$ and $\ell_0a_1-1<q^2$. We now choose a prime factor of $\ell_0 a_1-1$ that is different to $q$, say $p_0$, then $p_0$ is coprime to $qa_1$, so $p_0\nmid 1+(q-\ell_0)a_1$. This together with Chinese Remainder Theorem implies that we can choose $a_2$ such that
    \begin{align}\label{7-19-2}
    \begin{cases}
    a_2\equiv -1 \pmod{a_1},\\
    a_2(1+(q-\ell_0)a_1)\equiv -1\pmod{p_0}.
    \end{cases}
    \end{align}
    Recall that
    $x_{\ell_0}=\frac{a_2+1}{a_1}+(q-\ell_0)a_2$
    from (\ref{7-19-1}). We deduce that
    $p_0\mid x_{\ell_0}$ from (\ref{7-19-2}).

    Now we continue the construction of $a_2$ such that $\gcd(x_\ell, y_\ell)\neq 1$ for all $\ell=1,2,\dots,q$.
    We will do it by induction. If $p_0\mid a_1-1=y_1$, then we claim that $p_0\mid x_1$. In fact, Since $p_0\mid \ell_0 a_1-1$ and $p_0|a_1-1$, we have $p_0\mid \ell_0-1$ from which we deduced that
    $$
    p_0\mid a_2+1+(q-\ell_0)a_1a_2+(\ell_0-1)a_1a_2,
    $$
    by combining with (\ref{7-19-2}).
    Noting that
    $$
    a_2+1+(q-\ell_0)a_1a_2+(\ell_0-1)a_1a_2=a_2+1+(q-1)a_1a_2=a_1x_1,
    $$
    we conclude that $p_0\mid x_1$. If $p_0\nmid y_1$, then we choose a prime factor of $y_1$, say $p_1$. Since $p_1$ is coprime to $qa_1$, so $p_1$ is coprime to $1+(q-1)a_1$, and hence by Chinese Remainder Theorem, we can choose $a_2$ such that
    \begin{align}\label{7-19-3}
    \begin{cases}
    a_2\equiv -1\pmod{a_1},\\
    a_2(1+(q-\ell_0)a_1)\equiv -1\pmod{p_0},\\
    a_2(1+(q-1)a_1)\equiv -1\pmod{p_1}.
    \end{cases}
    \end{align}
    By the second congruence of (\ref{7-19-3}) we have $p_1\mid x_1$.

    Repeating the procedure above, suppose that we have chosen suitable $a_2$ such that $p_i\mid \gcd(x_i,y_i)$ for $i=1,2,\dots, \ell-1$. It is worth mentioning that $p_i$ may not be different here. We consider the case $\ell\neq \ell_0$. If $y_{\ell}$ is divided by some $p_i$ for $i\in\{0,1,\dots,\ell-1\}$, then we put $p_{\ell}=p_i$ and by the same reason as above, we have $p_{\ell}\mid x_{\ell}$. If $y_{\ell}$ is coprime to all $p_0,p_1,\dots,p_{\ell-1}$, then we choose $p_{\ell}$ to be a prime factor of $y_{\ell}$. By our construction of $\ell_0$, we see that $p_{\ell}$ is coprime to $pa_1$, so $p_\ell\nmid 1+(q-\ell)a_1$. Then by Chinese Remainder Theorem, we can choose $a_2$ such that $p_{i}\mid x_{i}$ for all $1\le i\le \ell$. Therefore, we would find out a suitable $a_2$ satisfying our requirement by induction on $\ell$.

    Since $q$ is fixed, such procedure will stop in finite steps, and by our construction, we have $p_{\ell}\mid \gcd(x_{\ell},y_{\ell})$ for all $\ell=1,2,\dots,q$, where $p_{\ell_0}=p_0$. At last, one notices from the prime number theorem that
    \[
    a_2\le a_1p_1p_2\dots p_{q}\leq a_1\prod_{p\le a_1q}p= a_1e^{(1+o(1))a_1q},
    \]
    where the $p$'s in the product represent primes.
    Hence, we have
    $$
    q\ge \big(1+o(1)\big)\frac{\log a_2-\log a_1}{a_1}\gg \log a_2,
    $$
    where the implied constant depends at most on $a_1$, proving our theorem.
\end{proof}

\section{Proof of Theorem \ref{thm2}, related results and unsolved problems}
\begin{proof}[Proof of Theorem \ref{thm2}]
    Since $4a-2$ can be only written as
    $$
    4a-2=2\cdot (a-1)+a\cdot 2=2\cdot(2a-1)+a\cdot 0,
    $$ we see that $4a-2$ can not be written as $2x_1+ax_2$ with $x_1,x_2\in \mathbb{Z}_{\geq 0}$ and $\gcd(x_1,x_2)=1$, that is $G_{2,a}\geq 4a-2$. On the other hand, for any $n>4a-2$, if $n$ is odd, then
    $$
    n=2\cdot \frac{n-a}{2}+a\cdot 1,
    $$
    is an admissible expression.
    If $n\equiv 2\pmod{4}$, then
    $$
    n=2\cdot \frac{n-4a}{2}+a\cdot 4,
    $$
    is admissible.
    If $n\equiv 0\pmod{4}$, then
    $$
    n=2\cdot \frac{n-2a}{2}+a\cdot 2,
    $$
    is an admissible expression.
\end{proof}

Let $k$ be a given positive integer. We are now interested in the prime powers $p^k\le g_{a_1,a_2}$ of the form 
$$
a_1x_1+a_2x_2 \quad (x_1,x_2\in \mathbb{Z}_{\ge 0}).
$$
Let $1<a_1<a_2$ be integers with $\gcd(a_1,a_2)=1$. Extending the result of Ding, Zhai and Zhao \cite{DingZhaiZhao}, recently Huang and Zhu \cite{HuangZhu} proved
$$
\pi_{k,a_1,a_2}\!:=\#\Big\{p^k\le g_{a_1,a_2}: p^k=a_1x_1+a_2x_2, x_1,x_2\in \mathbb{Z}_{\ge 0}\Big\}\sim \frac{k}{k+1}\frac{(g_{a_1,a_2})^{1/k}}{\log g_{a_1,a_2}},
$$
as $a_1\rightarrow \infty$. One notices from their result that $\pi_{k,a_1,a_2}>0$ provided that $a_1$ is sufficiently large. The result of Dai, Ding and Wang \cite{DZZ} (i.e., the solution of Conjecture \ref{conjecture1}) showed that
$
\pi_{1,a_1,a_2}=0
$
only for the pairs $(a_1,a_2)=(2,3)$.
In view of Conjecture \ref{conjecture1}, one naturally considers a similar problem. We wish to determine all the pairs $(a_1,a_2)$ such that $\pi_{2,a_1,a_2}=0$. The following theorem reflects quite different features between the situations of $k=1$ and $k=2$.

\begin{thm}\label{Ding-Wang-Zhang-new}
For any nonnegative integer $g$ we have
\[\pi_{2,6,6g+5}=\pi_{2,8,8g+7}=\pi_{2,12,12g+11}=\pi_{2,24,24g+23}=0.\]
\end{thm}

\begin{proof}
    Let $1<a<b$ be two relatively prime integers and $p$ a prime number with
    $$
    p^2<ab-a-b.
    $$
    If there are nonnegative integers $x,y$ such that $p^2=ax+by$, then $y\leq a-2$. For the case $a=6,8,12,24$ and $b=6g+5,8g+7,12g+11,24g+23$ respectively, it is not hard to see that $p\geq 5$. Actually, for these cases we clearly have
    $$
    p^2\ge a+b\ge 11.
    $$
  By classifying modulo $24$, we know that
    \begin{align}\label{prime-1}
    p^2\equiv 1\pmod{24}.
    \end{align}
    On the other hand, we clearly have $b\equiv -1\pmod{a}$, from which it follows that
    \begin{align}\label{prime-2}
    p^2=ax+by\equiv by \equiv -y\not\equiv 1\pmod{a},
    \end{align}
    provided that $y\leq a-2$. Hence, we have $p^2\not\equiv 1\pmod{24}$ from (\ref{prime-2}) and $a\mid 24$, which is certainly a contradiction with (\ref{prime-1}).
\end{proof}

It is worth here mentioning that
$$
\pi_{2,40,71}=\pi_{2,40,239}=\pi_{2,40,391}=\pi_{2,40,431}=\pi_{2,40,751}=\pi_{2,40,791}=0.
$$
Mathematical experiments then indicate the following conjecture.

\begin{conjecture}\label{conj-19-1}
Let $a_2>a_1>40$ be two integers with $\gcd(a_1,a_2)=1$. Then we have
$$
\pi_{2,a_1,a_2}>0.
$$
Furthermore, there are only finitely many pairs $(a_1,a_2)$ such that $\pi_{2,a_1,a_2}=0$ apart from the ones given in Theorem \ref{Ding-Wang-Zhang-new}.
\end{conjecture}

We could further consider the pairs $(a_1,a_2)$ such that $\pi_{k,a_1,a_2}=0$ for any given $k$. Here, perhaps we have some more interesting problems involving $\pi_{k,a_1,a_2}$. Let $g(k)$ be the least positive integer such that for any pair $(a_1,a_2)$ with $g(k)< a_1<a_2$ there is a prime power $p^k\le g_{a_1,a_2}$ satisfying $p^k=a_1x+a_2y,~x_1,x_2\in \mathbb{Z}_{\ge 0}$. The function $g(k)$ is well-defined, thanks to the theorem of
Huang and Zhu \cite{HuangZhu}.
Clearly we have $g(k)\ge (\sqrt{2})^k$.
We now pose a few problems below for further research.

\begin{problem}\label{prob-19-4}
Finding the (at least $\log$) asymptotic formula of $g(k)$ if it exists.
\end{problem}

\begin{problem}\label{prob-19-5}
Is it true that
$$
\lim_{k\rightarrow\infty}\frac{g(k+1)}{g(k)}=1?
$$
\end{problem}

\begin{problem}\label{prob-19-6}
Is it true that
$g(k+1)\ge g(k)$
for all sufficiently large $k$?
\end{problem}

It seems interesting to make the following conjecture.

\begin{conjecture}\label{conjecture-3-1}
Let $M>0$ be any given number. Then we have
$
g(k)>M^k
$
for all sufficiently large $k$.
\end{conjecture}

Let $1<a_1<a_2$ be integers with $\gcd(a_1,a_2)=1$. Another different perspective of this topic is the following problem. Let $\ell_{a_1,a_2}$ be the longest length of consecutive integers in the interval $[0,g_{a_1,a_2}]$ such that none of the which can be written as 
$$
a_1x_1+a_2x_2 \quad (x_1,x_2\in \mathbb{Z}_{\ge 0}).
$$
Clearly, we have $\ell_{a_1,a_2}=a_1-1$. In fact, none of the integers in $[1, a_1-1]$ has the desired expression. However, for any $m\ge 0$ the consecutive integers $m, m+1, \ldots, m+a_1-1$ contain a multiple of $a_1$, and this multiple of $a_1$ possesses the desired expression. Now, let  $L_{a_1,a_2}$ be the longest length of consecutive integers in the interval $[0,G_{a_1,a_2}]$ such that none of the whose elements can be written in the form 
\begin{align*}
n=a_1x_1+a_2x_2, \quad (x_1, x_2\in\mathbb{Z}_{\ge 0},~\gcd(x_1, x_2)=1).
\end{align*}
The following problem could be asked.
\begin{problem}\label{consecutive}
Finding the closed form of $L_{a_1,a_2}$.
\end{problem}

\section*{Acknowledgments}
We thank Changhao Chen, Zikang Dong, Nankun Hong, Honghu Liu, Yu-Chen Sun, Yitao Wu, Yanbin Zhang, Lilu Zhao and Tengyou Zhu for their interest in this article.

Y. D. is supported by National Natural Science Foundation of China  (Grant No. 12201544) and China Postdoctoral Science Foundation (Grant No. 2022M710121).
W. W is supported by  China
Postdoctoral Science Foundation (Grant No. 2024M763477). H. Z. is supported by Fundamental Research Funds for the Central Universities (Grant No. 531118010622), National Natural Science Foundation of China (Grant No. 1240011979) and Hunan Provincial Natural Science Foundation of China (Grant No. 2024JJ6120).

\end{document}